\documentclass[12pt]{amsart}
\usepackage{youngtab}
\usepackage{amsmath, color}
\usepackage{amssymb}
\usepackage{amsfonts}
\usepackage{amsthm}
\usepackage{mathrsfs}
\usepackage[all]{xy}

\newcommand{\be}{\begin{equation}}
\newcommand{\ee}{\end{equation}}

\newcommand{\la}{\lambda}
\newcommand{\om}{\omega}

\newtheorem{theorem}{Theorem}[section]

\newtheorem{proposition}[theorem]{Proposition}
\newtheorem{corollary}[theorem]{Corollary}

\theoremstyle{definition}
\newtheorem{definition}[theorem]{Definition}
\newtheorem{example}[theorem]{Example}

\theoremstyle{remark}
\newtheorem{remark}[theorem]{Remark}

\numberwithin{equation}{section}

\begin{document}

\dedicatory{Dedicated to Robert Griess, Jr. in honor of his 71st birthday}

\title[Fusion procedure for two-parameter quantum algebras]
{On fusion procedure for the two-parameter quantum algebra in type A} 
\author{Naihuan Jing}
\address{Department of Mathematics, North Carolina State University, Raleigh, NC 27695, USA
}
\email{jing@math.ncsu.edu}
\author{Ming Liu}
\address{School of Mathematics, South China University of Technology,
Guangzhou 510640, China}
\email{mamliu@scut.edu.cn}

\thanks{{\scriptsize
\hskip -0.4 true cm MSC (2010): Primary: 17B30; Secondary: 17B68.
}}

\maketitle

\begin{abstract}
Finite dimensional irreducible modules of the two-parameter quantum enveloping algebra
$U_{r,s}(\mathfrak{sl}_n)$ are explicitly constructed using the
fusion procedure when $rs^{-1}$ is generic. This provides an alternative
and combinatorial description of the Schur-Weyl duality for the two-parameter
quantum linear algebras of type $A$.
\end{abstract}

\section{Introduction}

Schur-Weyl duality is one of the main methods to construct irreducible
modules of the classical simple Lie groups out of the fundamental representations \cite{W}.
The quantum version for the quantum enveloping algebra
$U_q(\mathfrak{sl}_n)$ and the Hecke algebra $H_q({\mathfrak S}_m)$ has played an important role
in the fervent development of quantum groups. They provide one of the first examples \cite{Jb2} to show the
similarity between the classical and quantum theories.

Two-parameter general and special linear quantum groups \cite{T, DPW, DD}
are further generalization of the corresponding one-parameter Drinfeld-Jimbo quantum groups
\cite{D, Jb1}. The two-parameter quantum groups also had their origin
in the quantum inverse scattering method \cite{R} as
well as other approaches \cite{J, Do}. 
In particular,
the Schur-Weyl duality  was also generalized to the two-parameter case \cite{BW3}.
As in the classical situation, let $V$ be the natural $n-$ dimensional module of $U_{r,s}(\mathfrak{sl}_n)$, then arbitrary irreducible modules can be
constructed using the $(r,s)$-symmetric tensor
$S^2_{r,s}(V)$ and the R-matrix $R=R_{VV}$ for two-parameter quantum group
$U_{r,s}(\mathfrak{sl}_n)$. Alternatively any finite dimensional irreducible $U_{r,s}(\mathfrak{sl}_n)$-module
can be built from the tensor product $V^{\otimes m}$ using the symmetry
of the Hecke algebra $H_q({\mathfrak S}_m)$, where $q=\sqrt{s/r}$.

In this work we give an alternative
description of all irreducible representations ($rs^{-1}$ is not a root of unity)
using the developments \cite{IMO} of the fusion procedure \cite{C} in
the quantum inverse scattering method. It seems that the two-parameter
case can be treated quite similarly using the fusion procedure, thus one can more or
less apply the known results from the one-parameter case to get
corresponding formulas. As there is an abstract argument
available to construct the irreducible modules in the
two-parameter case, we nevertheless give a detailed
description of all irreducible $U_{r, s}(\mathfrak{sl}_n)$-modules
using the fusion procedure to supplement the existing theory.

Much of the results in the paper are expected for the experts, and  we hope the
current presentation can further show the similarity and connection to
the one-parameter case.
The one-parameter case is adapted
into the two parameter situation in a self-contained manner and
we try to be complete as much as possible for pedagogical purpose.

\section{Two-parameter quantum group $U_{r,s}(\mathfrak{sl}_n)$ and R-matrix}

We start with the basic definition of the two-parameter quantum groups $U_{r,s}(\mathfrak{gl}_n)$,
$U_{r,s}(\mathfrak{sl}_n)$ following the notations in \cite{BW3}. Let
$\Pi=\{\alpha_j=\epsilon_j-\epsilon_{j+1}|j=1,2,...,n-1\}$ be the
set of the simple roots of type $A_{n-1}$, where
$\epsilon_1,\epsilon_2,...,\epsilon_n$ are an orthonormal basis of a
Euclidean space with inner product $\langle\  ,\  \rangle$.
The root system is then
$\Phi=\{\epsilon_i-\epsilon_j|1\leq i\neq j\leq n\}$.

Fix two distinct nonzero complex numbers $r,s$, and assume they are in general position.

\begin{definition}
The two-parameter quantum enveloping algebra
$\widetilde{U}=U_{r,s}(\mathfrak{gl}_n)$ is the unital associative
algebra over $\mathbb{C}$ generated by $e_j,f_j$, $(1\leq j<n)$ and $a_i^{\pm 1}, {b}_i^{\pm 1}
$, $(1\leq i\leq n)$ with the following relations:

(R1) $a_i^{\pm1}$, $b_i^{\pm1}$ commute with each other, and $a_ia_i^{-1}=b_ib_i^{-1}=1$,

(R2) $a_ie_ja_i^{-1}=r^{\langle \epsilon_i,\alpha_j\rangle}e_j$ and
$a_if_ja_i^{-1}=r^{-\langle \epsilon_i,\alpha_j\rangle}f_j,$

(R3) $b_ie_jb_i^{-1}=s^{\langle \epsilon_i,\alpha_j\rangle}e_j$ and
$b_if_jb_i^{-1}=s^{-\langle \epsilon_i,\alpha_j\rangle}f_j,$

(R4) $[e_i,f_j]=\frac{\delta_{ij}}{r-s}(a_ib_{i+1}-a_{i+1}b_i)$,

(R5) $[e_i,e_j]=[f_i,f_j]=0, \mbox{if}\  |i-j|>1,$

(R6) $e_{i}^2e_{i+1}-(r+s)e_ie_{i+1}e_i+rse_{i+1}e_{i}^2=0$ and

     $e_{i}e^2_{i+1}-(r+s)e_{i+1}e_{i}e_{i+1}+rse_{i+1}^2e_{i}=0$,

(R7)
$f_{i}^2f_{i+1}-(r^{-1}+s^{-1})f_if_{i+1}f_i+r^{-1}s^{-1}f_{i+1}f_{i}^2=0$
and

$f_{i}f^2_{i+1}-(r^{-1}+s^{-1})f_{i+1}f_{i}f_{i+1}+r^{-1}s^{-1}f_{i+1}^2f_{i}=0$,
\newline where $[x\, ,\, y]=xy-yx$ is the commutator.

\end{definition}

The algebra $U=U_{r,s}(\mathfrak{sl}_n)$ is the subalgebra of $\widetilde{U}=U_{r,s}(\mathfrak{gl}_n)$
 generated by the elements $e_j$, $f_j$, $\omega_j$ and $\omega'_j$ $(1\leq j<n)$, where
 $\omega_j=a_jb_{j+1}$, and $\omega'_j=a_{j+1}b_j$. These elements satisfy the relations
 (R5)-(R7) along with

$(R1')$ The $\omega_i$, $\omega_j'$ are invertible and they all commute with each other,

$(R2')$
$\omega_ie_j=r^{\langle \epsilon_i,\alpha_j\rangle}s^{\langle \epsilon_{i+1},\alpha_j\rangle}e_j\omega_i$
and
$\omega_if_j=r^{-\langle \epsilon_i,\alpha_j\rangle}s^{-\langle \epsilon_{i+1},\alpha_j\rangle}f_j\omega_i,$

$(R3')$ $\omega_i'e_j=r^{\langle \epsilon_{i+1},\alpha_j\rangle}s^{\langle \epsilon_{i},\alpha_j\rangle}e_j\omega_i'$
and
$
\omega_i'f_j=r^{-\langle \epsilon_{i+1},\alpha_j\rangle}s^{-\langle \epsilon_{i},\alpha_j\rangle}f_j\omega_i',$

$(R4')$ $[e_i,f_j]=\frac{\delta_{ij}}{r-s}(\omega_i-\omega_i')$.

Clearly when $r=q$, $s=q^{-1}$, the algebra $U$ modulo the
ideal generated by the elements $\omega_j^{-1}-\omega_j'$, $1\leq
j<n$, is isomorphic to $U_q(\mathfrak{sl}_n)$.

The algebra $U_{r,s}(\mathfrak{sl}_n)$ is a Hopf algebra under the
comultiplication $\Delta$ such that $\omega_i$, $\omega_i'$ are group-like
elements and the other nontrivial comultiplications, counits and antipodes
are given by:

$$\Delta(e_i)=e_i\otimes 1+\omega_i\otimes e_i, \Delta(f_i)=1\otimes f_i+f_i\otimes \omega_i',$$

$$\epsilon(e_i)=\epsilon(f_i)=0, S(e_i)=-\omega_i^{-1}e_i, S(f_i)=-f_i\omega_i'^{-1}.$$

The representation theory of $U_{r,s}(\mathfrak{sl}_n)$ is quite similar to that of one-parameter
case. We recall some of the basic notations for later discussion. Let $\Lambda=\mathbb{Z}\epsilon_1\oplus
\mathbb{Z}\epsilon_2\oplus\cdots\oplus \mathbb{Z}\epsilon_n$ be the
weight lattice of $\mathfrak{gl}_n$, $Q=\mathbb{Z}\Phi$ the root lattice, and
$Q_+=\sum_{i=1}^{n-1}\mathbb{Z}_{\geq 0}\alpha_i$, where $\epsilon_i$ are the
orthonormal vectors as before. Recall that $\Lambda$
is equipped with the partial order given by $\nu\leq \lambda$ if and
only if $\lambda-\nu \in Q_+$.

For each $\lambda\in \Lambda$ one defines the
algebra homomorphism $\widehat{\lambda}: \widetilde{U}^0\rightarrow \mathbb{C}$ by:
\begin{equation}
\widehat{\lambda}(a_i)=r^{\langle\epsilon_i,\lambda\rangle},
~~~~~~~~~~~~~~~~~~~~~\qquad
~~~~~~~~~~~~~~~~~~~~~~\widehat{\lambda}(b_i)=s^{\langle\epsilon_i,\lambda\rangle},
\end{equation}
where $\widetilde{U}^0$ is the subalgebra of $U_{r,s}(\mathfrak{gl}_n)$ generated by $a_i^{\pm 1}$, $b^{\pm 1}_{i}$ $(1\leq i\leq n)$.
Then the restriction of $\widehat{\lambda}: U^0\rightarrow \mathbb{C}$ of $\widehat{\lambda}$ to the subalgebra $U^0$ of $U$ generated
by $\omega_j$, $\omega_j'$ $(1\leq j<n)$ satisfies:
\begin{equation}
\widehat{\lambda}(\omega_j)=r^{\langle\epsilon_j,\lambda\rangle}s^{\langle\epsilon_{j+1},\lambda\rangle},
~~~~~~~~~~~~~~~~~~~~~\qquad
~~~~~~~~~~~~~~~~~~~~~~\widehat{\lambda}(\omega_j')=r^{\langle\epsilon_{j+1},\lambda\rangle}s^{\langle\epsilon_i,\lambda\rangle}.
\end{equation}

It was shown in \cite{BW2} when $rs^{-1}$ is not a root of unity, the homomorphisms
$\widehat{\lambda}=\widehat{\mu}$ if and only if the corresponding
weights $\lambda=\mu$. These homomorphisms are called generalized weights.
For an algebra homomorphism $\chi: U^0\mapsto \mathbb
C$ one defines the generalized weight subspace of $U_{r, s}$-module M by
$$M_{\chi}=\{v\in M|
(\omega_i-\chi(\omega_i))^mv=(\omega_i'-\chi(\omega_i'))^mv=0, \,
\mbox{for all $i$ and for some $m$}\}.
$$

If $m=1$, the subspace $M_{\chi}$ becomes a weight subspace associated
with the homomorphism $\chi$.
Since $U^0$ is commutative, it is easy to see that
any finite dimensional
$U_{r,s}$-module $M$ can be decomposed into a sum of generalized weight subspaces:
\begin{equation}
M=\bigoplus_{\chi} M_{\chi}.
\end{equation}

When all generalized
weights in $M$ are of the form $\chi.(-\widehat{\alpha})$ for a fixed $\chi$
and $\alpha$ varying in $Q_+$ (here $\chi.(-\widehat{\alpha})(\omega_i)=\chi(\omega_i)(-\widehat{\alpha})(\omega_i)$
and $\chi.(-\widehat{\alpha})(\omega_i')=\chi(\omega_i')(-\widehat{\alpha})(\omega_i')$),
we say $M$ is a {\it highest weight
module} of weight $\chi$ and write $M=M(\chi)$. Benkart and
Witherspoon \cite{BW3} have shown that when $M$ is simple, all
generalized weight subspaces are actually weight subspaces.
Moreover,
if all generalized weights 
are of the form $\widehat{\lambda}$ for $\lambda\in \Lambda$, we will simply write
$M_{\lambda}$ for $M_{\widehat{\lambda}}$, and similarly the highest
weight module $M(\widehat{\lambda})$ will be denoted as
$M(\lambda)$.

One can also
define the notion of Verma modules $M(\lambda)$ as in the classical situation. It is known \cite{BW3} that all finite dimensional simple
$U_{r,s}(\mathfrak{sl}_n)$-modules are realized as simple
quotients of Verma modules. We will write by $V(\lambda)$ the
simple quotient of the Verma module $M(\lambda)$.

Now let us look at the simplest irreducible module $V(\overline{\omega}_1)$.
Let $V$ be the $n$-dimensional vector space over $\mathbb{C}$ with
basis $\{v_j|1\leq j\leq n\}$, and $E_{ij}\in End(V)$ be defined by
$E_{ij}v_k=\delta_{jk}v_i$. Define the $U_{r,s}(\mathfrak{sl}_n)$-action on $V$ by
\begin{equation*}
\begin{aligned}
&e_j=E_{j,j+1}, ~~f_j=E_{j+1,j}, \\
&\om_j=rE_{jj}+sE_{j+1, j+1}+\sum_{k\neq j, j+1} E_{kk},\\
&\om_j'=sE_{jj}+rE_{j+1, j+1}+\sum_{k\neq j, j+1} E_{kk}, \\
\end{aligned}
\end{equation*}
where $1\leq j\leq n-1$.
It is clear that
$V=\bigoplus_{j=1}^{n}V_{\epsilon_j}$ and it is the simple $U_{r,s}(\mathfrak{sl}_n)$-module $V(\overline{\omega}_1)$.

Let $\check{R}=\check{R}_{VV}: V\otimes V\longrightarrow V\otimes V$ be the
$R$-matrix defined by
\begin{equation}\label{R-matrix}
\check{R}=\sum_{i=1}^nE_{ii}\otimes E_{ii}+r\sum_{i<j}E_{ji}\otimes
E_{ij}+s^{-1}\sum_{i<j}E_{ij}\otimes
E_{ji}+(1-rs^{-1})\sum_{i<j}E_{jj}\otimes E_{ii},
\end{equation}
which is essentially determined by the simple module $V$ and the comultiplication $\Delta$
(see Prop. \ref{e:SW1}) .
For each $1\leq i<k$ let $\check{R}_i$ be the
isomorphism on $V^{\otimes k}$ defined by
\begin{equation*}
\breve{R}_i(w_1\otimes w_2\otimes\cdots\otimes
w_k)=w_1\otimes\cdots\otimes \check{R}(w_i\otimes w_{i+1})\otimes
w_{i+2}\otimes\cdots\otimes w_k.
\end{equation*}
Then we have the braid relations:
\begin{equation}\label{e:braid}
\check{R}_i\check{R}_{i+1}\check{R}_i=\check{R}_{i+1}\check{R}_i\check{R}_{i+1}
~~~~~~ for ~~~~1\leq i<k-1,
\end{equation}
The construction also implies that for $|i-j|\geq 2$,
\begin{equation*}
\check{R}_i\check{R}_j=\check{R}_j\check{R}_i.
\end{equation*}

Moreover one can directly check that
\begin{equation}\label{minimal polynomial}
\check{R}_i^{2}=(1-\frac{r}{s})\check{R}_i+\frac{r}{s}Id
\end{equation}
for all $1\leq i<k$. In particular, the minimum polynomial of $\check{R}$ on $V\otimes V$ is
$(t-1)(t+\frac{r}{s})$ if $s\neq -r$ \cite{BW3}.

\begin{proposition} \cite{BW3} \label{e:SW1}
The endomorphisms $\check{R}_i\in End(V^{\otimes k})$ commute with
the action of $U_{r,s}(\mathfrak{sl}_n)$ on $V^{\otimes k}$.
\end{proposition}
\begin{proof} By the Hopf algebra structure of $U_{r,s}(\mathfrak{sl}_n)$ it is enough
to check this for $k=2$. Then it is a direct verification to confirm that $\check{R}$ commutes with
$\Delta(e_i), \Delta(f_i),
\Delta(\om_i), \Delta(\om_i')$ for $i=1, \cdots, n-1$ on the fundamental representation
$V(\overline{\omega}_1)$.
\end{proof}

\section{Yang-Baxterization and the wedge modules of $U_{r,s}(\mathfrak{sl}_n)$}
The fusion procedure relies on the spectral parameter dependent R-matrix $\check{R}(z)$,
which satisfies the so-called quantum Yang-Baxter equation. One form of the Yang-Baxter equation (YBE) is the following
matrix equation on $V^{\otimes 3}$:
\begin{equation}\label{e:YBE}
\check{R}_1(z)\check{R}_2(zw)\check{R}_1(w)=\check{R}_2(w)\check{R}_1(zw)\check{R}_2(z).
\end{equation}
where $\check{R}(0)=\check{R}$. The Yang-Baxterization method recovers the spectral parameter dependent
R-matrix $\check{R}(z)$ from the initial condition (\ref{e:braid}) at $z=0$.

The Yang-Baxterization process was carried out for the two-parameter $R$-matrix in \cite{JZL} using the method of \cite{GWX}.

\begin{proposition}\cite{JZL}
For the braid group representation $\breve{R}=\check{R}_{VV}$, the
R-matrix $\check{R}(z)$ is given by
\begin{align}\nonumber
\check{R}(z)&=(1-zrs^{-1})\sum_{i=1}^nE_{ii}\otimes
E_{ii}+(1-z)(r\sum_{i>j}+s^{-1}\sum_{i<j})E_{ij}\otimes E_{ji}\\
\label{R(z)} &+z(1-rs^{-1})\sum_{i<j}E_{ii}\otimes
E_{jj}+(1-rs^{-1})\sum_{i>j}E_{ii}\otimes E_{jj}.
\end{align}
\end{proposition}

\begin{remark} Clearly $\check{R}(0)=\check{R}$. Moreover, when $r=q$ and $s=q^{-1}$, the R-matrix
$\check{R}(z)$ turns into
\begin{equation*}
\check{R}_q(z)=(1-zq^2)\sum_{i=1}^nE_{ii}\otimes E_{ii}+(1-z)q\sum_{i \neq
j}E_{ij}\otimes E_{ji}+(1-q^2)(\sum_{i>j}+z\sum_{i<j})E_{ii}\otimes
E_{jj},
\end{equation*}
which is exactly the Jimbo R-matrix for the quantum affine algebra
$U_{q}(\mathfrak{sl}_n)$ \cite{Jb1}. In this regard we can view $\check{R}(z)$ as an
$(r,s)$-analogue of the R-matrix $\check{R}_q(z)$ of the quantum affine
algebra $U_{q}(\widehat{\mathfrak{sl}}_n)$ \cite{HRZ}. The one-parameter
$R$-matrix can also be used to treat the quantum algebra $U_q(\mathfrak{gl}(m|n))$ \cite{Da}.
\end{remark}

For convenience, we introduce a normalized R-matrix with two spectral parameters
\begin{align}\nonumber
\check{R}(x,y)&=\frac{sy\check{R}(x/y)}{y-x}\\ 
&=\frac{sy-rx}{y-x}\sum_{i=1}^nE_{ii}\otimes E_{ii}
+(sr\sum_{i>j}+\sum_{i<j})E_{ij}\otimes E_{ji}  \label{e:YBE2}\\
&+\frac{(s-r)x}{y-x}\sum_{i<j}E_{ii}\otimes
E_{jj}+\frac{(s-r)y}{y-x}\sum_{i>j}E_{ii}\otimes E_{jj}. \nonumber
\end{align}
The original YBE immediately implies the following Yang-Baxter equation:
\begin{equation}\label{e:braid2}
\check{R}_i(x,y)\check{R}_{i+1}(x,z)\check{R}_i(y,z)
=\check{R}_{i+1}(y,z)\check{R}_i(x,z)\check{R}_{i+1}(x,y).
\end{equation}

Let $S_{r,s}^2(V)$ be the subspace of $V\otimes V$ spanned by
$\{v_i\otimes v_i|1\leq i\leq n\}\cup \{v_i\otimes v_j+sv_j\otimes v_i|1\leq i<j\leq n\}$
and $\Lambda_{r, s}^2(V)$ be the subspace of $V\otimes V$ spanned by
$\{v_i\otimes v_j-rv_j\otimes v_i|1\leq i<j\leq n\}$.
\begin{proposition}\cite{JZL}\label{prop3.3} The subspace
$S_{r,s}^2(V)$ is equal to the image of $\check{R}(1, r^{-1}s)$ on $V\otimes
V$, and $\Lambda_{r, s}^2(V)$ is equal to the image of $\check{R}(1, rs^{-1})$
on $V\otimes V$.
\end{proposition}

\begin{remark} The above result is equivalent to
$S_{r,s}^2(V)=Ker\, \check{R}(1, sr^{-1})$ and
$\Lambda_{r, s}^2(V)=Ker\, \check{R}(1, s^{-1}r)$. This suggests that special values
of the Yang-Baxter matrix can lead to irreducible modules, which we will show in
general in Section \ref{fusion}.
\end{remark}

\begin{proposition}\cite{JZL} The $k$th fundamental representation of $U_{r,s}(\mathfrak{sl}_n)$
can be realized as the following quotient of the $k$-fold tensor product
\begin{equation}
V^{\otimes k}/\sum_{i=0}^{k-2}V^{\otimes i}\otimes S^2_{r,s}(V)
\otimes V^{\otimes (k-i-2)}\cong V({\overline{\omega}_k}).
\end{equation}
\end{proposition}
In Section \ref{fusion} we will give an alternative way to construct all irreducible modules using
the Yang-Baxter equation.

\section{The Hecke algebra and the Schur-Weyl duality for $U_{r,s}(\mathfrak{gl}_n)$}

For any $r, s\in\mathbb{C}\setminus\{0\}$, we introduce the Hecke algebra $H_{m}(r,s)$ as follows.
\begin{definition} The Hecke algebra
$H_m(r,s)$ is the unital associative algebra over $\mathbb{C}$ with
generators $T_i$, $1\leq i<m$, subject to the following relations:

(H1): $T_iT_{i+1}T_{i}=T_{i+1}T_iT_{i+1},~~~~~~~1\leq i<m-1,$

(H2): $T_iT_j=T_jT_i,~~~~~~~|i-j|>1,$

(H3): $(T_i-1)(T_i+\frac{r}{s})=0. $
\end{definition}

\begin{remark}
When $r\neq 0$, the elements $t_i=\sqrt{\frac{s}{r}}T_i$ satisfy

$$(H3'): (t_i-\sqrt{\frac{s}{r}})(t_i+\sqrt{\frac{r}{s}})=0.$$

If we set $q=\sqrt{\frac{s}{r}}$, the two parameter Hecke algebra
$H_m(r,s)$ is isomorphic to $H_m(q)$, the Hecke algebra associated to the
symmetric group ${\mathfrak S}_m$. By a well-known result of Hecke algebras
$H_m(r,s)$ is semisimple whenever
$\sqrt{\frac{s}{r}}$ is not a root of unity.
\end{remark}

From Section 2, it is easy to verify that the $U_{r,s}(\mathfrak{gl}_n)$-module
$V^{\otimes m}$ affords a representation of Hecke algebra $H_m(r,s)$:
\begin{align}
H_m(r,s)\rightarrow End_{U_{r,s}(\mathfrak{gl}_n)}(V^{\otimes m})
\end{align}
$$T_i\mapsto \check{R}_i~~~~~~~~~~~~~~~~(1\leq i<m).$$

Benkart and Witherspoon \cite{BW3} gave a two parameter analogue of
the Schur-Weyl duality for $U_{r,s}(\mathfrak{sl}_n)$ and
the Hecke algebra $H_{r, s}({\mathfrak S}_k)$ associated with the
symmetric group ${\mathfrak S}_k$, which we recall as follows. 

\begin{proposition} \cite{BW3}
Assume $rs^{-1}$ is not a root of unity. Then:

(i) $H_{m}(r,s)$ maps surjectively onto $End_{U_{r,s}(\mathfrak{gl}_n)}(V^{\otimes m})$.

(ii) if $n\geq m$, $H_m(r,s)$ is isomorphic to $End_{U_{r,s}(\mathfrak{gl}_n)}(V^{\otimes m})$.
\end{proposition}
Let $\lambda=(\la_1, \la_2, \ldots, \la_l)$ be a partition, where $l=l(\lambda)$ is the length.
\begin{corollary}\label{Cor4.4}
When $rs^{-1}$ is not a root of unity,
as an $(U_{r,s}(\mathfrak{gl}_n)\otimes H_m(r,s))$-module, the space $V^{\otimes m}$ has
the following decomposition:
\begin{equation*}
V^{\otimes m}\cong \bigoplus_{\lambda} V_{\lambda}\otimes V^{\lambda},
\end{equation*}
where the partition $\lambda$ of $m$ runs over the set of partition such that
$l(\lambda)\leq n$,
$V_{\lambda}$ is the $U_{r,s}(\mathfrak{gl}_n)$-module associated to $\lambda$,
and $V^{\lambda}$ is the $H_m(r,s)$-module corresponding to $\lambda$.
\end{corollary}
\begin{remark}
The above theorem is a two-parameter version of the well-known
result of Jimbo \cite{Jb2}.
\end{remark}
In section \ref{fusion} we will give detailed information on the idempotents, which will
then give a realization of all irreducible modules.

\section{The orthogonal primitive idempotents of $H_{m}(r,s)$}
For any index $i=1,...,m-1$, let $s_i=(i,i+1)$ be the adjacent
transposition in the symmetric group $\mathfrak{S}_m$. Take any
element $\sigma\in \mathfrak{S}_m$ and choose a reduced
decomposition $\sigma=s_{i_1}s_{i_2}...s_{i_k}$. Denote
$T_{\sigma}=T_{i_1}T_{i_2}...T_{i_k}$, this element in $H_m(r,s)$
does not depend on the reduced decompositions of $\sigma$ \cite{IMO}.

The Jucys-Murphy elements of $H_m(r,s)$ are defined inductively by
\begin{equation}
y_1=1,~~~~~~~~~~~~~~~~~~~~~~~~~y_{k+1}=\frac{s}{r}T_ky_kT_k,
\end{equation}
where $k=1,...,m-1$. These elements satisfy
$$y_kT_l=T_ly_k,~~~~~~~~~~~~~~l\neq k,k-1.$$
In particular, $y_1, y_2,...y_m$ generate a commutative subalgebra
of $H_m(r,s)$. For any $k=1,...,m$, we let $w_k$ denote the
unique longest element of symmetric group $\mathfrak{S}_k$, which is
regarded as a natural subgroup of $\mathfrak{S}_m$. The
corresponding elements $T_{w_k}\in H_m(r,s)$ are given by
$T_{\omega_1}=1$ and
\begin{equation}
\begin{aligned}
T_{w_k}&=T_1(T_2T_1)...(T_{k-2}T_{k-3}...T_1)(T_{k-1}T_{k-2}...T_1)\\
&=(T_{1}...T_{k-2}T_{k-1})(T_{1}...T_{k-3}T_{k-2})...(T_1T_2)T_1,~~~~~~~~~~~~~~~k=2,...,m.
\end{aligned}
\end{equation}
It can be verified easily that
\begin{equation}\label{eq5.3}
T_{w_k}T_j=T_{k-j}T_{w_k}~~~~~~~~~~~~~~~1\leq j<k\leq m,
\end{equation}
\begin{equation}
T_{w_k}^2=(\frac{r}{s})^{\frac{k(k-1)}{2}}y_1y_2...y_k~~~~~~~~~~~~~~~~~, k=1,...,m.
\end{equation}

For each $i=1,...m-1$, we define the elements \cite{N}:
\begin{equation}
T_i(x,y)=sT_i+\frac{s-r}{\frac{y}{x}-1},
\end{equation}
where $x,y$ are complex variables. We will regard $T_i(x,y)$ as a
rational functions in $x,y$ with values in $H_m(r,s)$. These
functions satisfy the braid relations:
\begin{equation}
T_i(x,y)T_{i+1}(x,z)T_i(y,z)=T_{i+1}(y,z)T_i(x,z)T_{i+1}(x,y).
\end{equation}


We will identify a partition
$\lambda=(\lambda_1,\lambda_2,...,\lambda_l)$ of $m$ with its Young
diagram. The Young diagram is a left-justified array of rows of cells such that the
first row contains $\lambda_1$ cells, the second row contains
$\lambda_2$ cells, etc. A cell outside $\lambda$ is called addable
to $\lambda$ if the union of $\lambda$ and the cell is a (proper) Young diagram. A
$\lambda-tableau$ is obtained by filling in the cells of the Young diagram
bijectively with the numbers $1,...,m$. A tableau $T$ is called
standard if the entries of the tableau increase along the rows and
down the columns. If a cell occurs in the (i,j)-th position, its
(r,s)-content will be defined as $(\frac{s}{r})^{j-i}$. Let $\sigma_k$
denote the (r,s)-content of the cell occupied by $k$ in $T$.
\begin{example}
For $\lambda=(2,1)$, the corresponding Young diagram is
 $$\yng(2,1).$$
The cells at the positions $(1,3)$, $(2,2)$, $(3,1)$ are the addable cells. For the
$\lambda$-tableau
$T=\scriptsize\young(13,2)$, its $(r,s)$-contents are
$\sigma_1=1$, $\sigma_2=\frac{r}{s}$, $\sigma_3=\frac{s}{r}$.
\end{example}

A set of primitive idempotents of $H_m(r,s)$ parameterized by
partitions $\lambda$ of $m$ and the standard $\lambda$-tableaux $T$
can now be defined inductively by the following rule \cite{DJ}. Set
$E_T^{\lambda}=1$ if $m=1$, whereas for $m\geq 2$, one defines inductively
\begin{equation}\label{Eq5.7}
E_T^{\lambda}=E_U^{\mu}\frac{(y_m-\rho_1)...(y_m-\rho_k)}{(\sigma-\rho_1)...(\sigma-\rho_k)},
\end{equation}
where $U$ is the tableau obtained from $T$ by removing the cell
$\alpha$ occupied by $m$, $\mu$ is the shape of $U$, and
$\rho_1,...,\rho_k$ are the (r,s)-contents of all the addable cells
of $\mu$ except for $\alpha$, while $\sigma$ is the (r,s)-content of
the latter.

These elements satisfy the characteristic property that if $\lambda$ and $\lambda'$ are
partitions of $m$
\begin{equation}
E_T^{\lambda}E_{T'}^{\lambda'}=\delta_{\lambda,\lambda'}\delta_{T,T'}E_T^{\lambda}
\end{equation}
for arbitrary standard tableaux $T$ and $T'$ of shapes $\lambda$ and
$\lambda'$ respectively. Moreover,
\begin{equation}
\sum_{\lambda}\sum_{T}E_{T}^{\lambda}=1,
\end{equation}
summed over all partitions $\lambda$ of $m$ and all the standard
$\lambda$-tableaux $T$.

\section{Fusion formulas for the orthogonal primitive idempotents of $H_m(r,s)$}\label{fusion}
 We now apply the fusion
formulas \cite{IMO} to the situation of the two-parameter quantum algebra to derive a corresponding
formula for the idempotents of $H_m(r,s)$, which
can then be used to construct all the irreducible $U_{r,s}(\mathfrak{sl}_n)$-modules.

Let $\lambda=(\lambda_1,...,\lambda_l)$ be a partition of $m$,
the conjugate partition $\lambda'=(\lambda_1',...,\lambda_{l'}')$ of
$\lambda$ the partition of $m$ whose diagram obtained by reflection in the main diagonal.
Hence $\lambda_i'$ is the number of nodes in the ith column of $\lambda$.
Define
$$b(\lambda)=\sum_{i\geq 1}{\lambda_i'\choose 2}.$$
If $\alpha=(i,j)$ is a cell of $\lambda$, then the corresponding
hook is defined as $h_{\alpha}=\lambda_i+\lambda_j'-i-j+1$.

Now we introduce the rational function in complex variables $u_1, \ldots, u_m$
with values in $H_m(r,s)$,
$$\Psi(u_1,...u_m)=\prod_{k=1,...,m-1}(T_k(u_1,u_{k+1})T_{k-1}(u_2,u_{k+1})...T_1(u_k,u_{k+1}))\cdot
T_{w_m}^{-1},$$
where the product is carried out in the order of $k=1, \cdots, m-1$. The following theorem is obtained
by a similar argument as in the one-parameter case \cite{IMO}.

\begin{theorem}
For the partition $\lambda$ of $m$ and a standard $\lambda$-tableau $T$,
the idempotents $E^{\lambda}_T$ can be obtained by the consecutive evaluations
\begin{equation}
E^{\lambda}_T=f(\lambda)\Psi(u_1,...u_m)|_{u_1=\sigma_1}|_{u_2=\sigma_2}...|_{u_m=\sigma_m},
\end{equation}
where
\begin{equation*}
f(\lambda)=(\frac{s}{r})^{b(\lambda')}s^{-{m\choose 2}
}(1-\frac{s}{r})^m\prod_{\alpha\in
\lambda}(1-(\frac{s}{r})^{h_{\alpha}})^{-1}.
\end{equation*}
\end{theorem}

\begin{example}\label{ex6.2}
For $m=2$ and $\lambda=(2)$ we get
$$E^{\lambda}_T=f(\lambda)(sT_1+r),$$
where $f(\lambda)=\frac{1}{r+s}$, for
the standard tableau
$T={\scriptsize\young(12)}$,
In particular, $\sigma_1=1$, $\sigma_2=\frac{s}{r}$.
\end{example}
\begin{example}\label{ex6.3}
For $m=2$ and $\lambda=(1,1)$, we get
$$E^{\lambda}_T=f(\lambda)\frac{s^2}{r}(1-T_1),$$
where $f(\lambda)=\frac{r}{s(r+s)}$, for
the standard tableau
$T={\scriptsize\young(1,2)}$,
In particular, $\sigma_1=1$, $\sigma_2=\frac{r}{s}$.
\end{example}

\begin{example}
For $m=3$, $\lambda=(1,1, 1)$, $T={\scriptsize\young(1,2,3)}$,
we have
\begin{equation*}
\Psi(u_1,u_2,u_3)|_{u_1=\sigma_1}|_{u_2=\sigma_2}|_{u_3=\sigma_3}=\frac{s^6}{r^3}(1-T_1-T_2+T_1T_2+T_2T_1-T_1T_2T_1),
\end{equation*}
and $f(\lambda)=\frac{r^3}{(s+r)(s^2+rs+r^2)s^3}$. Thus we get
\begin{equation*}
E^{\lambda}_T=\frac{s^3}{(s+r)(s^2+rs+r^2)}(1-T_1-T_2+T_1T_2+T_2T_1-T_1T_2T_1),
\end{equation*}
which is the same as that in formula (\ref{Eq5.7}).
\end{example}
From Corollary \ref{Cor4.4}, we have
\begin{equation*}
E^{\lambda}_T(V^{\otimes m})\cong E^{\lambda}_T(\bigoplus_{\mu} V_{\mu}\otimes V^{\mu}).
\end{equation*}
Since
$E_T^{\lambda}$ annihilates the irreducible
$H_m(r,s)$-module $V^{\lambda'}$, we have
\begin{equation*}
E^{\lambda}_T(V^{\otimes m})\cong E^{\lambda}_T(\bigoplus_{\mu} V_{\mu}\otimes V^{\mu})\cong V_{\lambda}\otimes E_{T}^{\lambda}V^{\lambda}.
\end{equation*}
Furthermore, since $E_T^{\lambda}$ acts on the irreducible module $V^{\lambda}$ of
$H_m(r,s)$ as a projector on a 1-dimensional subspace,
we can get the following explicit description of the irreducible
modules of $U_{r, s}(\mathfrak{sl}_n)$.
\begin{theorem}\label{thm6.5}
For a partition $\lambda=(\lambda_1,...,\lambda_l)$ of $m$ with length $l\leq n$ and $T$ a standard $\lambda$-tableau,
$$V(\lambda)=E^{\lambda}_T(V^{\otimes m})$$
is the finite dimensional irreducible representation of $U_{r,s}(\mathfrak{sl}_n)$ with the highest weight
$\sum_{i=1}^{n-1}(\lambda_{i}-\lambda_{i+1})\overline{\omega}_i$.
\end{theorem}
\begin{remark}
We can see that Proposition \ref{prop3.3} is a special case of Theorem \ref{thm6.5}.
Actually, it is easy to check that $\check{R}(1,r^{-1}s)=s(s\check{R}+r)$. While from
Example \ref{ex6.2} for $\lambda=(2)$, we have $$E^{\lambda}_T=\frac{1}{r-s}(sT_1+r).$$
Thus we have $S_{r,s}^2(V)=\check{R}(1,r^{-1}s)V^{\otimes 2}=E^{(2)}_T(V^{\otimes 2})$.
Similarly, using Example \ref{ex6.3} we have $\Lambda_{r, s}^2(V)=\check{R}(1,rs^{-1})V^{\otimes 2}=E^{(1,1)}_T(V^{\otimes 2})$.

\end{remark}
\medskip

\centerline{\bf Acknowledgments} We thank Alex Molev for discussions on fusion procedures.
NJ gratefully acknowledges the
partial support of Simons Foundation grant no. 523868, Humbolt Foundation,
NSFC grant no. 11531004,
``Mathematical Methods from Physics'' in Summer of 2013 at KITPC,
and MPI for Mathematics in the Sciences in Leipzig during this work.
Liu thanks the support of NSFC grant no. 11701182.

\bibliographystyle{amsalpha}

\end{document}